\documentclass{amsart}

\title[Explicit $7$-torsion in the Tate-Shafarevich groups of genus $2$ Jacobians]{Explicit $7$-torsion in the Tate--Shafarevich\\ groups of genus $2$ Jacobians}

\author{Sam Frengley} 
\address{School of Mathematics, University of Bristol, Bristol, BS8 1UG, UK}
\email{sam.frengley@bristol.ac.uk}
\date{23 September 2025}

\usepackage[a4paper,margin=2.5cm]{geometry}
\usepackage{colonequals}	
\usepackage{amsmath}		
\usepackage{amssymb}		
\usepackage{amsthm}
\usepackage{thmtools}
\usepackage{amsfonts}		
\usepackage{enumitem}		
\usepackage[pagebackref=true,hypertexnames=false]{hyperref}

\def\MR#1{\href{http://www.ams.org/mathscinet-getitem?mr=#1}{MR#1}}

\usepackage{mathrsfs}
\usepackage{graphicx}
\usepackage{caption}
\usepackage{mathdots}
\usepackage{mathtools}
\usepackage{exercise}
\usepackage{csvsimple}
\usepackage{array,booktabs}
\usepackage{listings}
\usepackage{color}
\usepackage{epstopdf}
\usepackage{subcaption}
\usepackage{float}
\usepackage{pdfpages}
\usepackage{etoolbox}
\usepackage[utf8]{inputenc}
\usepackage{pdflscape}

\usepackage{cleveref}

\usepackage{tikz}
\usepackage{tikz-cd}

\usepackage{array}

\usepackage{makecell}
\usepackage{adjustbox}
\usepackage{multirow}

\renewcommand{\arraystretch}{1.5}

\DeclareFontFamily{U}{wncy}{}
\DeclareFontShape{U}{wncy}{m}{n}{<->wncyr10}{}
\DeclareSymbolFont{mcy}{U}{wncy}{m}{n}
\DeclareMathSymbol{\Sha}{\mathord}{mcy}{"58} 

\newcommand{\bbF}{\mathbb{F}}

\newcommand{\bbP}{\mathbb{P}}
\newcommand{\bbQ}{\mathbb{Q}}

\newcommand{\bbZ}{\mathbb{Z}}

\newcommand{\ffp}{\mathfrak{p}}

\newcommand{\overbar}[1]{\mkern 1.5mu\overline{\mkern-1.5mu#1\mkern-1.5mu}\mkern 1.5mu}

\newcommand{\Kbar}{\overbar{K}}

\newcommand{\Qbar}{\overbar{\mathbb{Q}}}

\newcommand{\LMFDBLabel}[1]{\textnormal{\href{https://www.lmfdb.org/EllipticCurve/Q/#1/}{\texttt{#1}}}}
\newcommand{\LMFDBLabelGenusTwo}[1]{\textnormal{\href{https://www.lmfdb.org/Genus2Curve/Q/#1}{\texttt{#1}}}}

\DeclareMathOperator{\Gal}{Gal}
\DeclareMathOperator{\End}{End}

\DeclareMathOperator{\Aut}{Aut}

\DeclareMathOperator{\GL}{GL}
\DeclareMathOperator{\SL}{SL}
\DeclareMathOperator{\PSL}{PSL}
\DeclareMathOperator{\PGL}{PGL}

\DeclareMathOperator{\Jac}{Jac}

\newcommand{\eniii}{\textnormal{(\roman*)}}

\usepackage{chngcntr}
\counterwithin{table}{section}

\usepackage{algorithm}
\usepackage{algorithmic}

 

\newtheorem{theorem}[algorithm]{Theorem}
\newtheorem{prop}[algorithm]{Proposition}
  
\newtheorem{lemma}[algorithm]{Lemma}
\newtheorem{conj}[algorithm]{Conjecture}
\newtheorem*{theorem*}{Theorem}
\newtheorem*{lemma*}{Lemma}
\newtheorem*{prop*}{Proposition}
\newtheorem*{coro*}{Corollary}  
\theoremstyle{definition}\newtheorem{defn}[algorithm]{Definition}
\theoremstyle{definition}\newtheorem{remark}[algorithm]{Remark}
\theoremstyle{definition}
\theoremstyle{definition}
\theoremstyle{definition}
\theoremstyle{definition}
\theoremstyle{definition}\newtheorem*{defn*}{Definition}  

\newcommand{\UpsN}[1]{\Gamma\!_{#1}}
\DeclareMathOperator{\Nm}{Nm}

\begin{document}

\begin{abstract}
  Let $C/\bbQ$ be a genus $2$ curve whose Jacobian $J/\bbQ$ has real multiplication by a quadratic order in which $7$ splits. We describe an algorithm which outputs twists of the Klein quartic curve which parametrise elliptic curves whose mod $7$ Galois representations are isomorphic to a sub-representation of the mod $7$ Galois representation attached to $J/\bbQ$. Applying this algorithm to genus $2$ curves of small conductor in families of Bending and Elkies--Kumar we exhibit a number of genus $2$ Jacobians whose Tate--Shafarevich groups (unconditionally) contain a non-trivial element of order $7$ which is visible in an abelian three-fold.
\end{abstract}

\maketitle

\section{Introduction}
\newcommand{\condbound}{500\,000}
Let $K$ be a number field and let $A/K$ be an abelian variety. For each place $v$ of $K$ we denote the completion of $K$ at $v$ by $K_v$. We write $G_K = \Gal(\Kbar/K)$ for the absolute Galois group of $K$ and write $G_v = \Gal(\overbar{K_v}/K_v)$. The Tate--Shafarevich group of $A/K$ is the group
\begin{equation*}
  \Sha(A/K) = \ker \left( H^1(G_K, A) \to  \prod_v H^1(G_{v}, A) \right)
\end{equation*}
where $v$ ranges over places of $K$. The non-trivial elements of the group $\Sha(A/K)$ parametrise torsors for $A/K$ which have $K_v$-rational points for every place $v$, but no $K$-points. In this article we prove the following theorem.

\begin{theorem}
  \label{coro:lots-of-sha}
  If $C/\bbQ$ is one of the genus $2$ curves in \Cref{tab:sha-exs}, then the Jacobian $J = \Jac(C)$ of $C$ is absolutely simple (i.e., $J$ is not isogenous over $\Qbar$ to a product of elliptic curves), has conductor at most $(500\,000)^2$, and the Tate--Shafarevich group $\Sha(J/\bbQ)$ contains a subgroup isomorphic to $(\bbZ/7\bbZ)^2$.
\end{theorem}

\begin{remark}
  The genus $2$ Jacobian $J/\bbQ$ of conductor $3200^2$ in \Cref{tab:sha-exs} was included in an appendix to \cite{KS_CVOSBSDFMMASOQ} joint with Keller and Stoll where the strong Birch and Swinnerton-Dyer conjecture is also verified for $J$.
\end{remark}

\begingroup
\renewcommand*{\arraystretch}{1.2}
\begin{table}[t]
  \centering
  \begin{tabular}{c >{\centering\arraybackslash}p{9.8cm} c c}
    $D$  & $f(x)$                                                                         & $\sqrt{N_J}$      & $E$                     \\
    \hline
    $8$  & $-10( x^{6} -4 x^{5} -3 x^{4} +8 x^{3} +25 x^{2} +20 x +5)$                    & $\texttt{3200}$   & \LMFDBLabel{3200.a1}    \\
    $8$  & $165( x^{6} +6 x^{5} +27 x^{4} -2 x^{3} +45 x^{2} +20)$                        & \texttt{39325}    & \LMFDBLabel{39325.c1}   \\
    $37$ & $-13( 27 x^6 - 54 x^5 - 90 x^4 + 228 x^3 + 15 x^2 - 90 x - 23 )$              & \texttt{73008}    & \LMFDBLabel{73008.n1}   \\
    $8$  & $-51( x^{6} +6 x^{5} +27 x^{4} -2 x^{3} +45 x^{2} +20)$                        & \texttt{93925}    & \LMFDBLabel{93925.d1}   \\
    $8$  & $285( x^{6} +6 x^{5} +27 x^{4} -2 x^{3} +45 x^{2} +20)$                        & \texttt{117325}   & \LMFDBLabel{117325.c1}  \\
    $8$  & $-62(9 x^{6} -12 x^{5} +64 x^{4} -56 x^{3} +136 x^{2} -60 x +84)$              & $\texttt{184512}$ & \LMFDBLabel{184512.bw1} \\
    $8$  & $-46( x^{6} +6 x^{5} -20 x^{4} +240 x^{3} +70 x^{2} -84 x +12)$                & $\texttt{203136}$ & \LMFDBLabel{203136.i2}  \\
    $8$  & $-5(3 x^{6} +12 x^{5} +89 x^{4} -56 x^{3} -7 x^{2} -132 x +99)$                & $\texttt{211200}$ & \LMFDBLabel{211200.c1}  \\
    $8$  & $465(9 x^{6} -12 x^{5} +64 x^{4} -56 x^{3} +136 x^{2} -60 x +84)$              & \texttt{216225}   & \LMFDBLabel{432450.ci1} \\
    $8$  & $-30(11 x^{6} -18 x^{5} +47 x^{4} +6 x^{3} +71 x^{2} +18 x +27)$               & $\texttt{244800}$ & \LMFDBLabel{244800.dc1} \\
    $8$  & $-13( x^{6} -2 x^{5} +3 x^{4}  x^{3} -7 x^{2} -2 x +1)$                        & $\texttt{256880}$ & \LMFDBLabel{51376.e1}   \\
    $8$  & $-390( x^{6} +6 x^{5} +27 x^{4} -2 x^{3} +45 x^{2} +20)$                       & $\texttt{270400}$ & \LMFDBLabel{270400.dc2} \\
    $8$  & $-177(3 x^{6} +12 x^{4} -10 x^{3} -12 x +11)$                                  & \texttt{281961}   & \texttt{2819610*}       \\
    $8$  & $-22( x^{6} -24 x^{5} +100 x^{4} +102 x^{3} -80 x^{2} -132 x -39)$             & $\texttt{302016}$ & \LMFDBLabel{302016.p1}  \\
    $8$  & $-6(13 x^{6} -116 x^{5} -316 x^{4} +58 x^{3} +264 x^{2} -116 x +13)$           & $\texttt{313920}$ & \LMFDBLabel{313920.bb1} \\
    $8$  & $-55(13 x^{6} -116 x^{5} -316 x^{4} +58 x^{3} +264 x^{2} -116 x +13)$          & \texttt{329725}   & \texttt{659450*}        \\
    $8$  & $-110( x^{6} -4 x^{5} -3 x^{4} +8 x^{3} +25 x^{2} +20 x +5)$                   & $\texttt{387200}$ & \texttt{8905600*}       \\
    $8$  & $11( x^{6} +6 x^{5} +11 x^{4} -13 x^{2} +6 x -2)$                              & $\texttt{423984}$ & \LMFDBLabel{423984.by1} \\
  \end{tabular}
  \caption{Examples of genus $2$ curves $C \colon y^2 = f(x)$ whose Jacobians $J/\bbQ$ have conductor $N_J < (500\,000)^2$, and such that $\Sha(J/\bbQ)$ contains a subgroup isomorphic to $(\bbZ/7\bbZ)^2$ (we do not claim, nor expect, this list to be complete). The Jacobians have real multiplication by the quadratic order $\mathcal{O}_D$ of discriminant $D$ and the subgroup of $\Sha(J/\bbQ)$ is made visible by a $(7, \ffp)$-congruence between $E/\bbQ$ and $J/\bbQ$ where $\ffp$ divides $7$ in $\mathcal{O}_D$. We write $N*$ for an elliptic curve of conductor $N$ which does not appear in the LMFDB (i.e., if $N > 500\,000$). Explicit Weierstrass equations for the corresponding elliptic curves are given in \cite{ME_ELECTRONIC_here}. The curves $C/\bbQ$ were generated from \cite[Theorem~4.1]{B_COG2W2M} and \cite{EK_K3SAEFHMS,CFM_FMFG2CWRM}. Conductors were computed using~\cite{DD_3TACOG2C}.}
  \label{tab:sha-exs}
\end{table}
\endgroup

The group $\Sha(A/K)$ is a torsion group and is conjectured to be finite. It is conjectured that for every prime number $p > 0$ and each integer $g > 0$ there exists an absolutely simple abelian variety (i.e., one which is not isogenous over $\Qbar$ to a product) of dimension $g$ for which $\Sha(A/\bbQ)[p] \neq 0$. Indeed, it is even conjectured that $\Sha(A^d/\bbQ)[p] \neq 0$ for a positive proportion of quadratic twists of a \emph{fixed} abelian variety $A/\bbQ$ (see e.g., \cite[Conjecture~1.1]{BKOS_EOGOITSGOAVIQTF}). 

In spite of this, for general values of $g$ and $p$, constructing an example of an absolutely simple $g$-dimensional abelian variety $A/\bbQ$ with an $p$-torsion element contained in $\Sha(A/\bbQ)$ is an open problem. By allowing the dimension of $A$ to increase with $p$, Shnidman and Weiss~\cite{SW_EOPOITSGOAVOQ} construct examples of absolutely simple abelian varieties with $\Sha(A/\bbQ)[p] \neq 0$. Flynn and Shnidman~\cite{FS_ALpTITSG} extended this result to show $\Sha(A/\bbQ)[p]$ can be arbitrarily large.

When $A$ has dimension $2$, Bruin, Flynn, and Testa~\cite{BFT_DV33IOJOG2C,F_DV55IOJOG2C} found examples of absolutely simple genus $2$ Jacobians with $3$ and $5$-torsion in their Tate--Shafarevich groups. Their approach relies on \emph{$(p,p)$-descent}. That is, for several explicit examples of genus $2$ curves $C/\bbQ$ they determined the $\psi$-Selmer groups of their Jacobians $J/\bbQ$ where $\psi$ is a $(p,p)$-isogeny (i.e., a polarised isogeny with kernel isomorphic to $(\bbZ/p\bbZ)^2$) for each $p = 3, 5$.

However, performing a $(p,p)$-descent becomes computationally costly as $p$ increases, due to the need to perform class and unit group calculations in (a subfield of) the field $\bbQ(J[\psi])$. Our approach is to instead leverage \emph{visibility} (see e.g., \cite{CM_VEITSTG,AS_VOSTGOAV,AS_VEFTBSDCFMAVOAR0,F_VEOO7ITTSGOAEC}) to construct absolutely simple genus $2$ Jacobians such that $\Sha(J/\bbQ)[7] \neq 0$.

Let $K$ be a number field and let $A/K$ and $A'/K$ be abelian varieties equipped with isogenies $\psi \colon A \to B$ and $\psi' \colon A' \to B'$.

\begin{defn}
  We say that $A/K$ and $A'/K$ are \emph{$(\psi, \psi')$-congruent} if there exists a $G_K$-equivariant group isomorphism $\phi \colon A[\psi] \to A'[\psi']$. We say that $\phi$ is a \emph{$(\psi, \psi')$-congruence}.

  In this case let $\Delta = \operatorname{Graph} \phi \subset A \times A'$. An element of $\Sha(A/K)$ is said to be \emph{visible} in the abelian variety $Z = (A \times A')/\Delta$ if it is contained in the kernel of the induced homomorphism $H^1(G_K, A) \to H^1(G_K, Z)$.
\end{defn}

Visibility is useful for constructing elements of $\Sha(A/K)$ since it allows us to transport information between the Mordell--Weil group of $B'/K$ and the Tate--Shafarevich group of $A/K$.  More precisely, if $B(K)/\psi A(K) = 0$ then, under mild hypotheses applied at the bad primes of $A$ and $A'$ and the primes dividing $| \Delta |$ (see~\cite[Theorem~2.2]{F_VEOO7ITTSGOAEC}), the group $\operatorname{Vis}_Z \Sha(A/K)$ of elements of $\Sha(A/K)$ that are visible in $Z = (A \times A')/\Delta$ is equal to $B'(K)/\psi' A'(K)$.

The central idea for proving \Cref{coro:lots-of-sha} is to construct examples of genus $2$ curves $C/\bbQ$ with the property that there exists a $(7, \psi)$-congruence between an elliptic curve $E/\bbQ$ and the Jacobian $J = \Jac(C)$ of $C$, for some isogeny $\psi \colon J \to B$. Assuming that the local conditions are satisfied, it then suffices to show that $B(\bbQ)/\psi J(\bbQ) = 0$ and that $E(\bbQ)/7E(\bbQ) \neq 0$ (which in practice is computationally less intensive than performing a $\psi$-descent on $J$). This approach is a mirror to that taken by Fisher~\cite{F_VEOO7ITTSGOAEC} who used it to visualise elements of order $7$ in the Tate--Shafarevich groups of elliptic curves. The main technical contribution of this article is construct examples of elliptic curves which are $(7,\psi)$-congruent to a genus $2$ Jacobian.

We ensure the existence of such an isogeny $\psi$ by choosing $J/\bbQ$ to have real multiplication (RM) by a real quadratic order $\mathcal{O}_D$ of fundamental discriminant $D > 0$. Suppose that $7$ splits in $\mathcal{O}_D$ and we have an embedding $\mathcal{O}_D \subset \End_\bbQ(J)$. Writing $(7) = \ffp \bar{\ffp}$ in $\mathcal{O}_D$ by abuse of notation we write $\ffp \colon J \to B$ for the isogeny with kernel consisting of those $P \in J(\Qbar)$ annihilated by $\ffp$. In this case, $\ker \ffp$ is isomorphic as a group to $(\bbZ/7\bbZ)^2$ and under suitable hypotheses (see~\Cref{lemma:split}) comes equipped with a natural alternating bilinear pairing.

In \Cref{sec:algo-twist-X7} we describe an algorithm for determining (a $q$-adic approximation to) a pair of twists of the Klein quartic (the modular curve $X(7)$) which parametrise elliptic curves that are $(7, \ffp)$-congruent to a fixed genus $2$ Jacobian $J/\bbQ$ with real multiplication by $\mathcal{O}_D$ (our algorithm is subject to the technical hypothesis that $J[\ffp]$ is an irreducible $G_{\bbQ}$-module). 

\begin{remark}
  Since abelian surfaces $J/\bbQ$ with RM by $\mathcal{O}_D$ are modular (this follows from Serre's conjecture~\cite{KW_SMCI,KW_SMCII}) we may associate to $J$ a weight $2$ newform with coefficients in $\mathcal{O}_D$ and level $\sqrt{N_J}$, where $N_J$ is the conductor of $J$ (in particular $N_J$ is a perfect square).
\end{remark}

We compute these twists of $X(7)$ for examples of genus $2$ Jacobians of small conductor provided by the real multiplication families of Bending~\cite{B_COG2W2M_thesis,B_COG2W2M} and of Elkies--Kumar~\cite{EK_K3SAEFHMS,CFM_FMFG2CWRM}. By searching for rational points on these twists, we find a number of putative examples of $(7, \ffp)$-congruences between an elliptic curve $E/\bbQ$ and a genus $2$ Jacobian $J/\bbQ$. Adapting an approach of Fisher~\cite[Section~6]{F_VEOO7ITTSGOAEC} we prove these congruences in \Cref{thm:sqrt2-congs}.

For an abelian variety $A/\bbQ$ we write $A^d/\bbQ$ for the quadratic twist of $A$ by a squarefree integer $d \in \bbZ$. Note that simultaneous quadratic twists of $(7,\ffp)$-congruent pairs remain $(7,\ffp)$-congruent (cf.~\cite[Lemma~4.15]{FK_OTSTOIOTPTOEC}). To construct the examples in \Cref{coro:lots-of-sha} we simply search for quadratic twists of the examples in \Cref{thm:sqrt2-congs} where there is a rank discrepancy between $E^d/\bbQ$ and $J^d/\bbQ$.

In addition to \Cref{coro:lots-of-sha} we prove that there exist examples of such genus $2$ Jacobians with $7$-torsion in their Tate--Shafarevich groups and with real multiplication by $\mathcal{O}_D$ for several fundamental discriminants $D > 0$.

\begin{theorem}
  \label{thm:sha7-discs}
  For each $D = 8$, $29$, $37$, $44$, and $57$ there exists an absolutely simple genus $2$ Jacobian $J/\bbQ$ with real multiplication by $\mathcal{O}_D$ such that $\Sha(J/\bbQ)$ contains a subgroup isomorphic to $(\bbZ/7\bbZ)^2$. Examples are furnished by the Jacobians of the curves $C \colon y^2 = f(x)$ given in \Cref{tab:sha-exs-discs}.
\end{theorem}

\begingroup
\renewcommand*{\arraystretch}{1.2}
\begin{table}[t]
  \centering
  \begin{tabular}{c >{\centering\arraybackslash}p{9.8cm} c c}
    $D$  & $f(x)$                                                           & $\sqrt{N_J}$       & $E$                   \\
    \hline
    $8$  & $-10( x^{6} -4 x^{5} -3 x^{4} +8 x^{3} +25 x^{2} +20 x +5)$      & \texttt{3200}      & \LMFDBLabel{3200.a1}  \\
    $29$ & $-2470 (8x^6 - 2x^5 + 68x^4 + 221x^3 + 122x^2 + 986x + 1588)$    & \texttt{40019200}  & \texttt{760364800*}   \\
    $37$ & $-39( x^{6} -45 x^{4} -68 x^{3} +504 x^{2} +180 x -1193)$        & \texttt{73008}     & \LMFDBLabel{73008.n1} \\
    $44$ & $-39 (14x^6 - 30x^5 + 85x^4 + 700x^3 - 1325x^2 + 3000x + 18000)$ & \texttt{608400}    & \texttt{608400*}      \\
    $57$ & $1479 (80x^6 + 279x^4 + 186x^3 +243x^2 + 324x +108)$             & \texttt{590609070} & \texttt{7677917910*}  \\
  \end{tabular}
  \caption{Examples of genus $2$ curves $C \colon y^2 = f(x)$ whose Jacobians $J/\bbQ$ have real multiplication by the quadratic order $\mathcal{O}_D$ of discriminant $D$ and such that $\Sha(J/\bbQ)$ contains a subgroup isomorphic to $(\bbZ/7\bbZ)^2$. The subgroup of $\Sha(J/\bbQ)$ is made visible by a $(7, \ffp)$-congruence between $E/\bbQ$ and $J/\bbQ$ where $\ffp$ divides $7$ in $\mathcal{O}_D$. We write $N*$ for an elliptic curve of conductor $N$ which does not appear in the LMFDB (i.e., if $N > 500\,000$). Explicit Weierstrass equations for the corresponding elliptic curves are given in \cite{ME_ELECTRONIC_here}. The curves $C/\bbQ$ were generated from \cite[Theorem~4.1]{B_COG2W2M} and \cite{EK_K3SAEFHMS,CFM_FMFG2CWRM}. Conductors were computed using~\cite{DD_3TACOG2C}.}
  \label{tab:sha-exs-discs}
\end{table}
\endgroup

In \cite[A.3]{KS_CVOSBSDFMMASOQ} it is observed that the Birch and Swinnerton--Dyer conjecture predicts that $\Sha(J^{-11}/\bbQ)$ contains a subgroup isomorphic to $(\bbZ/7\bbZ)^2$, where $J/\bbQ$ is the Jacobian of the genus $2$ curve with LMFDB label \LMFDBLabelGenusTwo{385641.a.385641.1}. Since $J$ has RM by $\mathcal{O}_8$ it is natural to ask whether the $7$-torsion in $\Sha(J^{-11}/\bbQ)$ is made visible by a $(7,\ffp)$-congruence with an elliptic curve. By computing the relevant twists of $X(7)$ we give evidence that this is not the case (see \Cref{sec:worthy-example}).

\begin{conj}
  \label{conj:sha-invisible}
  Let $C/\bbQ$ be the genus $2$ curve with LMFDB label \LMFDBLabelGenusTwo{385641.a.385641.1} and let $J/\bbQ$ be its Jacobian. There exists a subgroup isomorphic to $(\bbZ/7\bbZ)^2$ contained in $\Sha(J^{-11}/\bbQ)$ that is not visible in an abelian $3$-fold.
\end{conj}

\subsection{Outline of the paper}
\label{sec:outline-paper}
We begin by discussing several well known facts about the modular curve $X(p)$ in \Cref{sec:Xp}. In \Cref{sec:XMp}, following \cite[Section~4.4]{PSS_TOX7APSTX2Y3Z7}, we recall the moduli interpretation for twists of $X(p)$, which in \Cref{sec:XJp} we relate to the torsion of genus $2$ curves with Jacobians having real multiplication (cf.~\cite{F_VEOO7ITTSGOAEC}). In \Cref{sec:twist-princ-X7} we specialise to the case when $p = 7$ and discuss the invariant theory of the Klein quartic $X(7)$ following \cite{F_OFO7A11CEC}.

In \Cref{sec:algo-twist-X7} we present our main algorithm. It takes as input a genus $2$ Jacobian $J/\bbQ$ with RM by an order in which $7$ splits, and outputs four twists of $X(7)$ which parametrise elliptic curves $(7, \ffp)$-congruent to $J/\bbQ$.

The outputs of the algorithm in \Cref{sec:algo-twist-X7} are not guaranteed to be correct, however in \Cref{sec:proving-twists} we prove that the output is correct in many cases (for example for the curves in \Cref{tab:sha-exs,tab:sha-exs-discs}). In particular, in \Cref{sec:proving-twists} we prove that the twists we obtain are isomorphic to those which parametrise elliptic curves $(7,\ffp)$-congruent to $J$.

In \Cref{sec:proving-congs} we prove \Cref{coro:lots-of-sha,thm:sha7-discs} by proving that the pairs $(E,J)$ in \Cref{tab:sha-exs,tab:sha-exs-discs} are $(7,\ffp)$-congruent, and by checking that the local hypotheses in \cite[Theorem~2.2]{F_VEOO7ITTSGOAEC} are satisfied.

Finally in \Cref{sec:worthy-example} we give explicit examples of the Klein quartic twists for the Jacobian of the genus $2$ curve with LMFDB label \LMFDBLabelGenusTwo{385641.a.385641.1}. By searching for rational points on these twists, we give evidence towards \Cref{conj:sha-invisible}.

\subsection{Acknowledgements}
This article formed part of my PhD thesis~\cite[Chapter~6]{F_thesis}. I would like to thank my PhD supervisor Tom Fisher for many discussions and for extensive comments on previous versions of this article. I am grateful to Timo Keller for helpful correspondences and to Maria Corte-Real Santos and Ross Paterson for useful comments and discussions.

This work was supported by the Woolf Fisher and Cambridge Trusts, and by the Royal Society through C{\'e}line Maistret's Dorothy Hodgkin Fellowship. The code associated to this article is written in \texttt{Magma}~\cite{MAGMA} and \texttt{SageMath}~\cite{sagemath} and is publicly available in the GitHub repository~\cite{ME_ELECTRONIC_here}.

\section{The modular curve \texorpdfstring{$X(p)$}{X(p)} and its twists \texorpdfstring{$X_M^\pm(p)$}{X\_M(p)}}
We recall a number of standard facts about the modular curve $X(p)$ and its twists $X_M^\pm(p)$ following e.g., \cite{PSS_TOX7APSTX2Y3Z7} and \cite{F_OFO7A11CEC}. In the case when $p = 7$ the curve $X(7)$ is the Klein quartic~\cite{K_UDTSODEF} (see \cite{E_TKQINT} for a detailed discussion).

Let $K$ be a field of characteristic zero. A symplectic abelian group over $K$ is a pair $(M, e_M)$ where $M$ is a $G_K$-module equipped with a ($G_K$-equivariant) alternating, bilinear pairing $e_M \colon M \times M \to \Kbar^\times$. We equip $\mu_p \times \bbZ/p\bbZ$ with the natural alternating pairing $\langle (\zeta, n) , (\xi, m) \rangle = \zeta^m \xi^{-n}$.

\subsection{The modular curve \texorpdfstring{$X(p)$}{X(p)}}
\label{sec:Xp}
Let $E/K$ be an elliptic curve defined over a field $K$ of characteristic zero. If $p$ is a prime number, we equip $E[p]$ with the structure of a symplectic abelian group via the $p$-Weil pairing $e_{E,p} \colon E[p] \times E[p] \to \mu_p$.

Let $Y(p)/\bbQ$ denote the geometrically irreducible (non-compact) modular curve parametrising elliptic curves with full (symplectic) level $p$ structure. Explicitly, for each field $K/\bbQ$ the $K$-points on $Y(p)$ parametrise (isomorphism classes of) pairs $(E, \iota)$ where $E/K$ is an elliptic curve and $\iota \colon  \mu_p \times \bbZ/p\bbZ \cong E[p]$ is a $G_K$-equivariant isomorphism of symplectic abelian groups. Let $X(p)$ denote the smooth compactification of $Y(p)$.

The group $\UpsN{p}$ of symplectic automorphisms of $\mu_p \times \bbZ/p\bbZ$ acts naturally on $Y(p)$ on the right via $(E, \iota) \mapsto (E, \gamma\iota)$. As an abstract group $\UpsN{p}$ is isomorphic to $\SL_2(\bbZ/p\bbZ)$, but it comes equipped with a non-trivial action of $\Gal(\bbQ(\mu_p)/\bbQ)$. The matrix $\pm I$ acts trivially on $Y(p)$ and therefore the action of $\UpsN{p}$ factors through $\UpsN{p} / \{ \pm I \}$. This action extends to an action on $X(p)$, and the quotient realises the forgetful morphism $X(p) \to X(1)$ given by taking $j$-invariants.

\subsection{The twist \texorpdfstring{$X_M^\pm(p)$}{X\_M{\textasciicircum}r(p)}}
\label{sec:XMp}
Let $(M, e_M)$ be a symplectic abelian group over $K$. Let $r$ be an integer coprime to $p$ and suppose that there exists a $\Kbar$-isomorphism $\phi \colon M \cong \mu_p \times \bbZ/p\bbZ $ such that $\langle \phi(P), \phi(Q) \rangle = e_{M}(P, Q)^r$ for each $P, Q \in \mu_p \times \bbZ/p\bbZ$. 

By the twisting principle we may attach to $\phi$ a cohomology class $\xi \in H^1(G_K, \UpsN{p})$. We have an inclusion $\Gamma_p / \{\pm I\} \hookrightarrow \Aut(X(p))$ and therefore an induced map on cohomology $H^1(G_K, \UpsN{p}) \to H^1(G_K, \Aut(X(p)))$. The image of $\xi$ corresponds (again by the twisting principle) to a twist $X_M^r(p)$ of $X(p)$.

The following lemma is well known and follows by construction (cf. \cite[Section~4.4]{PSS_TOX7APSTX2Y3Z7}).

\begin{lemma}
  \label{lemma:XMr-moduli}
  For each field $L/K$ the $L$-rational points on $X_M^r(p)$ correspond to pairs $(E, \phi)$ where $E/L$ is an elliptic curve and $\phi \colon M \cong E[p]$ is an isomorphism of $G_L$-modules for which $e_{E,p}(\phi(P), \phi(Q)) = e_{M}(P, Q)^r$ for each $P, Q \in M$.
\end{lemma}

If $a$ is an integer coprime to $p$, pre-composing an isomorphism $\phi \colon M \cong E[p]$ with the multiplication-by-$a$-map on $M$ yields an isomorphism $\phi'$ for which $e_{E,p}(\phi'(P), \phi'(Q)) = e_{M}(P,Q)^{a^2 r}$. It therefore suffices to consider the class of $r$ in $(\bbZ/p\bbZ)^\times$ modulo squares. We write $X_M(p) = X_M^+(p)$ when $r$ is a square in $(\bbZ/p\bbZ)^\times$ and $X_M^{-}(p)$ when $r$ is not a square in $(\bbZ/p\bbZ)^\times$.

\subsection{The twist \texorpdfstring{$X^r_{J[\ffp]}(p)$}{X\_J{\textasciicircum}r(p)}}
\label{sec:XJp}
Let $C/K$ be a genus $2$ curve and let $J = \Jac(C)$ be the Jacobian of $C$. Let $\widehat{J}$ denote the dual of $J$ and equip $J$ with the canonical principal polarisation $\lambda \colon J \to \widehat{J}$ arising from the theta divisor. The principal polarisation $\lambda$ induces an involution on the endomorphism ring of $J$ known as the \emph{Rosati involution}. Precisely, if $\psi \in \End(J)$ then the Rosati involution is given by $\psi \mapsto \psi^{\dagger} = \lambda^{-1} \hat{\psi} \lambda$. Here $\hat{\psi} \colon \widehat{J} \to \widehat{J}$ denotes the dual isogeny.

Let $D \equiv 0,1 \pmod{4}$ be a positive non-square integer and let $\mathcal{O}_D$ be the quadratic ring of discriminant $D$. We say that $J$ has \emph{real multiplication (RM) by $\mathcal{O}_D$} if there exists an inclusion $\mathcal{O}_D \hookrightarrow \End^{\dagger}_{K}(J)$ where $\End_K^\dagger(J) \subset \End_K(J)$ is the subring of endomorphisms fixed by the Rosati involution.

The choice of principal polarisation $\lambda$ induces the alternating, bilinear $p$-Weil pairing $e_{J,p} \colon J[p] \times J[p] \to \mu_p$.

\begin{lemma}
  \label{lemma:split}
  Let $J/K$ be a genus $2$ Jacobian with RM by $\mathcal{O}_D$. Suppose that $p$ is a prime number such that $p$ splits as a product $(p) = \ffp\bar{\ffp}$ in $\mathcal{O}_D$. Then there exists an isomorphism of $G_K$-modules $J[p] \cong J[\ffp] \oplus J[\bar{\ffp}]$. Moreover, the $p$-Weil pairing $e_{J,p}$ restricts to an alternating pairing $J[\ffp] \times J[\ffp] \to \mu_p$, and likewise for $J[\bar{\ffp}]$.
\end{lemma}

\begin{proof}
  This is \cite[Lemma~3.4]{CFM_FMFG2CWRM}, cf. the proof of \cite[Proposition~6.1]{F_VEOO7ITTSGOAEC} when $\mathcal{O}_D = \bbZ[\sqrt{2}]$.
\end{proof}

Note that by \Cref{lemma:split} we may define the twists $X_{J[\ffp]}^\pm(p)$ and $X_{J[\bar{\ffp}]}^\pm(p)$ which (by \Cref{lemma:XMr-moduli}) parametrise elliptic curves which are $(p, \ffp)$-congruent (respectively $(p, \bar{\ffp})$-congruent) to the genus $2$ Jacobian $J/K$.

\subsection{Explicit twisting for \texorpdfstring{$X(7)$}{X(7)}}
\label{sec:twist-princ-X7}
Following \cite{F_OFO7A11CEC} we give an explicit description for the twists $X_M^\pm(7)$ as plane quartic curves. Recall that we write $\UpsN{7}$ for the automorphism group (scheme) of the symplectic abelian group $\mu_7 \times \bbZ/7\bbZ$. Following Klein~\cite{K_UDTSODEF,E_TKQINT,F_OFO7A11CEC} consider the representation $\SL_2(\bbZ/7\bbZ) \to \GL_3(\Kbar)$ which maps the generators $S = \left(\begin{smallmatrix} 0 & -1 \\ 1 & 0 \end{smallmatrix}\right)$ and $T = \left(\begin{smallmatrix} 1 & 1 \\ 0 & 1 \end{smallmatrix}\right)$ to
\begin{equation*}
  \frac{1}{\sqrt{-7}}
  \begin{pmatrix}
    \zeta_7 - \zeta_7^6 & \zeta_7^2 - \zeta_7^5 & \zeta_7^4 - \zeta_7^3 \\
    \zeta_7^2 - \zeta_7^5 & \zeta_7^4 - \zeta_7^3 & \zeta_7 - \zeta_7^6 \\
    \zeta_7^4 - \zeta_7^3 & \zeta_7 - \zeta_7^6 & \zeta_7^2 - \zeta_7^5 \\
  \end{pmatrix}
  \quad \text{and} \quad
  \begin{pmatrix}
    \zeta_7 & 0 & 0 \\
    0 & \zeta_7^4 & 0 \\
    0 & 0 & \zeta_7^2 \\
  \end{pmatrix}.
\end{equation*}
Composing with the natural $\Kbar$-isomorphism $\UpsN{7} \cong \SL_2(\bbZ/7\bbZ)$ gives a $G_K$-equivariant homomorphism $\rho \colon \UpsN{7} \to \GL_3(\Kbar)$. When $\UpsN{7}$ acts on $\bbP^2$ via ${\rho}$ the Klein quartic curve $X(7) \subset \bbP^2$ given by
\begin{equation*}
 X(7) : x_0^3x_1 + x_1^3x_2 + x_0 x_2^3=0 
\end{equation*}
is fixed.

Let $\phi \colon M \cong \mu_7 \times \bbZ/7\bbZ$ be a symplectic $\Kbar$-isomorphism. Let $\xi \colon G_{K} \to \GL_3(\Kbar)$ denote the cocycle obtained from $\phi$ by the twisting principle. Explicitly, $\xi$ may be taken to be the cocycle $\sigma \mapsto {\rho} ( \sigma(\phi) \phi^{-1} )$ where $\sigma(\phi) \colon M \cong \mu_7 \times \bbZ/7\bbZ$ is the $\Kbar$-isomorphism given by $P \mapsto \sigma ( \phi (\sigma^{-1} P))$. By Hilbert's Theorem 90 the cocycle $\xi$ is a coboundary, that is, there exists a matrix $A \in \GL_3(\Kbar)$ such that $\sigma(A^{-1}) A = \xi(\sigma)$ for every $\sigma \in G_K$.

\begin{lemma}
  \label{lemma:XM7-explicit}
  The curves $X_M(7)$ and $X_M^{-}(7)$ are isomorphic to the images of the morphisms $X(7) \to \bbP^2$ given by
  \begin{equation*}
    \mathbf{x} \mapsto A \mathbf{x} \quad \text{and} \quad \mathbf{x} \mapsto A^{-{T}}\mathbf{x}
  \end{equation*}
  respectively. Here $\mathbf{x}$ is a point in $\bbP^2$ written as a column vector and $A^{-{T}}$ denotes the inverse transpose of $A$.
\end{lemma}

\begin{proof}
  The proof is identical to \cite[Lemma~3.2]{F_OFO7A11CEC}.
\end{proof}

Finally, we note that it is simple to recover the moduli interpretation for a given twist of $X(7)$ following~\cite[Section~4.1]{F_OFO7A11CEC} (cf. \cite[Section~7.1]{PSS_TOX7APSTX2Y3Z7}).

\begin{lemma}
  \label{lemma:calX-moduli}
  Let $\mathcal{X}/K$ be a twist of the Klein quartic given by the vanishing of a homogeneous quartic polynomial $\mathcal{F}(x_0, x_1, x_2)$ in $\bbP^2$. The $j$-invariant map $\mathcal{X} \to X(1)$ is defined over $K$ and is given by $1728 \frac{c_4^3}{c_4^3 - c_6^2}$ where $c_4$ and $c_6$ are defined by
  \begin{equation*}
    \setlength\arraycolsep{2pt}
    D = \frac{-1}{54} \times \left|
      \begin{matrix}
        \frac{\partial^2 \mathcal{F}}{\partial x_0^2} & \frac{\partial^2 \mathcal{F}}{\partial x_0 \partial x_1} & \frac{\partial^2 \mathcal{F}}{\partial x_0 \partial x_2} \\[1em]
        \frac{\partial^2 \mathcal{F}}{\partial x_0 x_1} & \frac{\partial^2 \mathcal{F}}{\partial x_1^2} & \frac{\partial^2 \mathcal{F}}{\partial x_1 \partial x_2} \\[1em]
        \frac{\partial^2 \mathcal{F}}{\partial x_0 x_2} & \frac{\partial^2 \mathcal{F}}{\partial x_1 \partial x_2} & \frac{\partial^2 \mathcal{F}}{\partial x_2^2}
      \end{matrix}
    \right|,
    \quad
    c_4 = \frac{1}{9} \times
    \left|
      \begin{matrix}
        \frac{\partial^2 \mathcal{F}}{\partial x_0^2}            & \frac{\partial^2 \mathcal{F}}{\partial x_0 \partial x_1} & \frac{\partial^2 \mathcal{F}}{\partial x_0 \partial x_2} & \frac{\partial D}{\partial x_0 } \\[1em]
        \frac{\partial^2 \mathcal{F}}{\partial x_0 \partial x_1} & \frac{\partial^2 \mathcal{F}}{\partial x_1^2}            & \frac{\partial^2 \mathcal{F}}{\partial x_1 \partial x_2} & \frac{\partial D}{\partial x_1}  \\[1em]
        \frac{\partial^2 \mathcal{F}}{\partial a \partial x_2}   & \frac{\partial^2 \mathcal{F}}{\partial x_1 \partial x_2} & \frac{\partial^2 \mathcal{F}}{\partial x_2^2}            & \frac{\partial D}{\partial x_2}  \\[1em]
        \frac{\partial D}{\partial x_0}                & \frac{\partial D}{\partial x_1}                & \frac{\partial D}{\partial x_2}                & 0
      \end{matrix}
    \right|,
    \quad
    c_6 = \frac{1}{14} \times
    \left|
      \begin{matrix}  
        \frac{\partial \mathcal{F}}{\partial x_0}   & \frac{\partial \mathcal{F}}{\partial x_1}   & \frac{\partial \mathcal{F}}{\partial x_2} \\[1em]
        \frac{\partial D}{\partial x_0}   & \frac{\partial D}{\partial x_1}   & \frac{\partial D}{\partial x_2} \\[1em]
        \frac{\partial c_4}{\partial x_0} & \frac{\partial c_4}{\partial x_1} & \frac{\partial c_4}{\partial x_2}
      \end{matrix}
    \right|.  
  \end{equation*}
\end{lemma}

\section{Computing approximations to twists of \texorpdfstring{$X(7)$}{X(7)}}	
Fix a primitive $7^{\text{th}}$ root of unity $\zeta_7$, and let $J/\bbQ$ be the Jacobian of a genus $2$ curve $C/\bbQ$. Suppose that $J$ has RM by the quadratic order $\mathcal{O}_D$ of fundamental discriminant $D > 0$. Suppose that $7$ splits in $\mathcal{O}_D$ and that we have a factorisation $(7) = \ffp \bar{\ffp}$ (where $\bar{\ffp}$ denotes the conjugate of $\ffp$). We assume throughout this section that $J[\ffp]$ is an irreducible $G_\bbQ$-module.

Let $\mathcal{K} = J/\{\pm 1\}$ denote the Kummer surface of $J$. We identify $\mathcal{K}$ with a singular quartic surface in $\bbP^3$ by the embedding in \cite[(3.1.8)]{CF_PTAMAOCOG2}. Let $x_J$ denote the quotient morphism $J \to \mathcal{K}$. If $\psi \colon J \to J'$ is an isogeny we write $\mathcal{K}[\psi] = x_J(J[\psi])$.

\subsection{Hilbert's Theorem 90 is effective}
\label{sec:effect-hilb-theor}
To compute twists of $X(7)$ using \Cref{lemma:XM7-explicit} we will need to compute matrices which realise a given cocycle as a coboundary. Towards this, note that the standard proof of Hilbert's Theorem 90 is ``nearly'' effective. Indeed following \cite[Proposition~X.3]{S_LF} let $L/K$ be a finite extension of infinite fields, let $\xi \in H^1(\Gal(L/K), \GL_n(L))$ be a $1$-cocycle, and choose an element $c \in \GL_n(L)$. Define a matrix $b \in \mathrm{M}_n(L)$ by the Poincar{\'e} series
\begin{equation*}
  b = \sum_{\sigma \in \Gal(L/K)} \xi(\sigma) \sigma(c).
\end{equation*}
If $c$ is chosen so that $b$ is invertible, then $\xi(\sigma) = \sigma(b)^{-1} b$ and it is immediate that $\xi$ is a coboundary. If $L$ is infinite the linear independence of field embeddings guarantees the existence of $c$. Indeed, for a fixed cocycle $\xi$, the failure of $b$ to be invertible is a Zariski closed condition on the matrix $c$. In particular, for a generic choice of $c$, the matrix $b$ will be invertible. The proof therefore suggests an algorithm.

\begingroup
\begin{algorithm}[H]
  \caption{Generating matrices which realise Hilbert's Theorem~90.} \label{algo:hilb-90}
  \begin{flushleft}
    \textbf{Input:} A cocycle $\xi \in H^1(\Gal(L/K), \GL_n(L))$. \\
    \textbf{Output:} A matrix $b \in \GL_n(L)$ such that $\xi(\sigma) = \sigma(b)^{-1} b$ for each $\sigma \in \Gal(L/K)$.\\
  \end{flushleft}
  \begin{algorithmic}[1]
    \STATE \label{step:choose} Choose ``randomly'' a matrix $c \in \GL_n(L)$.
    \STATE Compute the matrix $b = \sum_{\sigma} \xi(\sigma) \sigma(c)$.
    \IF{$\det(b) \neq 0$}
    \RETURN $b$
    \ELSE
    \STATE \label{step:redo} Return to step 1.
    \ENDIF
  \end{algorithmic}
\end{algorithm}
\endgroup

\begin{remark}
  \label{rmk:good-matrix-c}
  In step~\ref{step:choose} of \Cref{algo:hilb-90} the user must choose a matrix $c \in \GL_n(L)$. In our application the extension $L/\bbQ$ will be a finite extension and we will have access to a LLL-reduced $\bbZ$-basis $\{a_1, ..., a_m\}$ for the ring of integers $\mathcal{O}_L \subset L$. We choose ``small'' elements of $L$ by generating a tuple $x_1,...,x_m \in \{0,\pm 1\}$ and considering the element $\sum_{i=1}^m x_i a_i \in \mathcal{O}_L$. This approach extends to choosing a matrix $c \in \GL_n(L)$ by choosing the $n^2$ entries as described (in practice, we also choose almost all $x_i$ to be equal to zero).  
\end{remark}

\begin{remark}
  In principle the iteration in Step~\ref{step:redo} in \Cref{algo:hilb-90} may be called many times. In practice, however, we have found very few instances when \Cref{algo:hilb-90} fails to terminate in one iteration. 
\end{remark}

\subsection{The main algorithm}
\label{sec:algo-twist-X7}
We write $L = \bbQ(\mathcal{K}[\ffp])$. Note that for generic $J/\bbQ$, we have an isomorphism of abstract groups $\Gal(L / \bbQ) \cong \GL_2(\bbZ/7\bbZ)/\{\pm 1\}$.
Suppose that we have degree $24$ polynomials $g_1(t)$, $g_2(t)$, $g_3(t) \in \bbQ[t]$ such that 
\begin{equation}
  \label{eqn:fi-set}
  \mathcal{K}[\ffp] \subset \{ (1 : \alpha_1 : \alpha_2 :  \alpha_3) \in \mathcal{K} :  g_1(\alpha_1) =  g_2(\alpha_2) =  g_3(\alpha_3) = 0 \} \cup \{(0:0:0:1)\}.
\end{equation}
In \Cref{sec:an-expl-descr} we discuss how the polynomials $g_i(t)$ may be computed.

We fix an auxiliary prime $q \neq 7$ not dividing the discriminant of $g_1(t)$ and at which $J$ has good reduction. Further suppose that the minimal polynomial of $\zeta_7$ is irreducible over $\bbF_q$ and that $L$ is equal to the splitting field of $g_1(t)$. Let $\mathfrak{q}$ be a prime of $L$ dividing $q$ and denote by $L_{\mathfrak{q}}$ and $l_{\mathfrak{q}}$ the completion and residue field of $L$ at $\mathfrak{q}$ respectively.

Our algorithm proceeds as follows:

\begin{enumerate}[label={\small\arabic*:}]
\item
  We compute the Galois group of $g_1(t)$ using \texttt{GaloisGroup} in \texttt{Magma}. This gives a group $G \subset S_{24}$ and a $G$-set $\{r_1, ..., r_{24}\}$ of ($\mathfrak{q}$-adic approximations to) the roots of $g_1(t)$ in $L_{\mathfrak{q}}$ such that $G$ gives the action of $\Gal(L/\bbQ)$ on the roots of $g_1(t)$ in $L$.
  
\item
  We compute (a $\mathfrak{q}$-adic approximation to) an embedding $\bbQ_q(\zeta_7) \hookrightarrow L_{\mathfrak{q}}$ by computing a $\mathfrak{q}$-adic approximation to $\zeta_7$.

\item
  We compute a $\bbZ/7\bbZ$-basis $\{\bar{P},\bar{Q}\}$ for $J(l_{\mathfrak{q}})[\ffp]$ such that $e_{J,7}(\bar{P}, \bar{Q}) = \zeta_7$ (via the \texttt{Magma} intrinsic \texttt{WeilPairing}). This uniquely determines a pair $x_J(P),x_J(Q) \in \mathcal{K}(L_{\mathfrak{q}})$ which reduce modulo $\mathfrak{q}$ to $x_J(\bar{P})$ and $x_J(\bar{Q})$ respectively.

\item
  Let $\phi \colon J[\ffp] \cong \mu_7 \times \bbZ/7\bbZ$ be the isomorphism given by $P \mapsto (\zeta_7, 0)$ and $Q \mapsto (1,1)$. We explicitly determine the 1-cocycle
  \begin{equation*}
    \xi \colon \Gal(L/\bbQ) \cong G \to \GL_3(\bbQ(\zeta_7))
  \end{equation*}
  given by $\sigma \mapsto \rho (\sigma(\phi)\phi^{-1})$ where $\rho$ is the representation given in \Cref{sec:twist-princ-X7}.
  
\item
  We compute (a $\mathfrak{q}$-adic approximation to) a matrix $A \in \GL_3(L) \subset \GL_3(L_{\mathfrak{q}})$ which realises $\xi$ as a coboundary, using \Cref{algo:hilb-90}.

\item
  We twist $X(7)$ by $A$ and $A^{-T}$ to obtain curves $\mathcal{X}^{\pm} \subset \bbP^2_{L_\mathfrak{q}}$. By \Cref{lemma:XM7-explicit} these twists are $\mathfrak{q}$-adic approximations to the twists $X_{J[\ffp]}^{\pm}(7) \subset \bbP^2_{\bbQ}$. After normalising each equation so that the first non-zero coefficient is equal to $1$, the curves obtained therefore have coefficients in $\bbQ$ (up to a numerical error).

\item
  We recognise the coefficients of the twists $\mathcal{X}^\pm$ as rational numbers using the LLL algorithm.

\item
  We minimise and reduce the models for $\mathcal{X}^\pm$ using the algorithm of Elsenhans--Stoll~\cite{ES_MOH}, which is implemented in \texttt{Magma} as \texttt{MinRedTernaryForm}.
\end{enumerate}

\begin{remark}
  Computing the matrix $A \in \GL_3(L_{\mathfrak{q}})$ must be done with some care in order to control coefficient explosion (and to minimise the $\mathfrak{q}$-adic precision we must carry throughout the calculation). In our implementation we assume that $J[\ffp]$ is an irreducible $G_{\bbQ}$-module and (heavily) rely on the following observation:

  Let $\alpha$ be a root of $g_1(t)$ (so that $L$ is the splitting field of $\bbQ(\alpha)$). Since the Galois module $J[\ffp]$ is irreducible, there exists a subfield $\bbQ \subset K \subset \bbQ(\alpha)$ (which is unique up to conjugacy) such that $\mathcal{K}[\ffp]$ contains a $\Gal(L/ K)$-stable ``cyclic subgroup'' (or more precisely, the image of a cyclic subgroup of $J[\ffp]$). From the properties of the Weil pairing we have $L =  \widetilde{K}(\zeta_7)$, where $\widetilde{K}$ is the Galois closure of $K$ in $L$. We first compute an LLL-reduced basis $\{1, k_1,..., k_7\}$ for $K/\bbQ$. The elements $k \zeta_7^i$ span $L/\bbQ$ where $0 \leq i \leq 6$ and $k$ ranges over the $\Gal(L/\bbQ)$-conjugates of $k_j$ for each $1 \leq j \leq 7$. We then apply \Cref{algo:hilb-90} noting \Cref{rmk:good-matrix-c}. 
\end{remark}

\subsection{Computing the polynomials \texorpdfstring{$g_i(t)$}{g\_i(t)}}
\label{sec:an-expl-descr}
It remains to describe how the polynomials $g_1(t)$, $g_2(t)$, and $g_3(t)$ which cut out $\mathcal{K}[\ffp]$ may be computed. Let $C/\bbQ$ be a genus $2$ curve given by a Weierstrass equation $C : y^2 = f(x)$ whose Jacobian $J/\bbQ$ has RM by $\mathcal{O}_D$, and such that $7$ splits in $\mathcal{O}_D$.

Our approach follows that of Fisher~\cite[Theorem~6.3]{F_VEOO7ITTSGOAEC}. Using the analytic Jacobian machinery in \texttt{Magma} (in particular the functions \texttt{AnalyticJacobian} and \texttt{EndomorphismRing}) we compute complex approximations to a $\ffp$-torsion divisor $\mathfrak{D} = (x_1, y_1) + (x_2, y_2) - (\infty^+ + \infty^-) \in J(\Qbar)$.

The model for the Kummer surface $\mathcal{K}$ of $J$ given in \cite[Chapter~3]{CF_PTAMAOCOG2} and maps $\mathfrak{D}$ to the point $(1 : x_1 + x_2 : x_1 x_2 : \beta_0) \in \mathcal{K}$ where $\beta_0 \in \bbQ(x_1,x_2,y_1,y_2)$ is the rational function in \cite[(3.1.4)]{CF_PTAMAOCOG2}. Using the LLL algorithm we compute polynomials $h_1(t)$, $h_2(t)$, and $h_3(t) \in \bbQ(t)$ which approximate the minimal polynomials of $x_1 + x_2$, $x_1 x_2$, and $\beta_0$ (in particular we utilise the \texttt{Magma} function \texttt{MinimalPolynomial}). Using the description of the multiplication-by-$n$-map on $\mathcal{K}$ given in \cite[Chapter~3]{CF_PTAMAOCOG2} it is simple to verify (unconditionally) that the polynomials $h_i(t)$ cut out a $7$-torsion point in $\mathcal{K}(\Qbar)$. Polynomials $g_i(t)$ which satisfy \eqref{eqn:fi-set} are then the product over the distinct polynomials $h_i(t)$ occurring for such divisors $\mathfrak{D}$.

When $D = 8$ we also have the following approach which avoids the numerical instability issues which can occur when using \texttt{AnalyticJacobian}.

\subsubsection{A numerically stable approach when $D = 8$}
\label{sec:when-d-8}
Fix an isomorphism $\mathcal{O}_8 \cong \bbZ[\sqrt{2}]$. The prime number $7$ is a norm from $\bbZ[\sqrt{2}]$ and we may write $(7) = (3 + \sqrt{2})(3 - \sqrt{2})$. Let $[\sqrt{2}] \colon J \to J$ denote the multiplication-by-$\sqrt{2}$-map on $J$. The morphism $[\sqrt{2}]$ is a Richelot isogeny and using the approach in \cite[Section~5.7]{N_DMATOJOHGC} (which is implemented in \cite{N_Electronic}) we determine explicit polynomials giving the morphism $[\sqrt{2}] \colon \mathcal{K} \to \mathcal{K}$ induced by the action of $\sqrt{2}$ on $J$.

\begin{remark}
  By interpolation it is not difficult to give an explicit morphism $\mathcal{K} \to \mathcal{K}$ realising the $\sqrt{2}$-action on the Jacobian of the generic member of the generic family of genus $2$ curves $\mathcal{C}/\bbQ(A,P,Q)$ provided by Bending~\cite{B_COG2W2M_thesis,B_COG2W2M}. We record explicit equations for this (generic) morphism in \cite{ME_ELECTRONIC_here}.
\end{remark}

Formulae for the multiplication-by-$3$-map $[3] \colon \mathcal{K} \to \mathcal{K}$ are given in \cite[Section~3.5]{CF_PTAMAOCOG2}. Note that $\mathcal{K}[3 + \sqrt{2}] \cup \mathcal{K}[3 - \sqrt{2}]$ is exactly the set $\{P \in \mathcal{K} : 3P = \sqrt{2} P \}$. By taking successive resultants (and fixing a choice of sign so that $\ffp = (3 \pm \sqrt{2})$) it is simple to compute polynomials $g_1(t), g_2(t), g_3(t) \in \bbQ[t]$ satisfying \eqref{eqn:fi-set}.

\subsection{Outputs of the main algorithm}
\label{sec:outp-main-algor}
We provide a \texttt{Magma} implementation of the algorithm described in \Cref{sec:algo-twist-X7}. The main non-trivial input in the algorithm is a genus $2$ curve $C/\bbQ$ with RM by an order $\mathcal{O}_D$ in which $7$ splits. The fundamental discriminants $D < 100$ for which this occurs are $D = 8$, $29$, $37$, $44$, $53$, $57$, $60$, $65$, $85$, $88$, $92$, and $93$.

A generic family of genus $2$ curves $C/\bbQ$ whose Jacobians have RM by $\mathcal{O}_8$ are given by Bending \cite{B_COG2W2M_thesis,B_COG2W2M}, who also records many examples of small conductor in \cite[Appendix~A]{B_COG2W2M_thesis}. Bending's family is given by a triple of parameters $A, P, Q \in \bbQ$. It is simple to search for further examples of small conductor (noting from \cite[Section~6.3]{B_COG2W2M_thesis} that it is often useful to specialise at $P \in \{\pm 1, \pm 1/2, \pm 1/3, \pm 1/5\}$). Combining these with the examples found in the LMFDB~\cite{lmfdb} we obtain a small (non-exhaustive) database of curves with RM by $\bbZ[\sqrt{2}]$ and whose Jacobians have conductor $\sqrt{N_J} \leq \condbound$ (these may be found in \cite{ME_ELECTRONIC_here}).

Similar generic families are provided for each $D = 8$, $29$, $37$, $44$, and $53$ in \cite{CFM_FMFG2CWRM} building on work of Elkies--Kumar~\cite{EK_K3SAEFHMS}, who compute the moduli of such curves for all fundamental discriminants $D < 100$. Some examples of curves with RM by $\mathcal{O}_D$ and with small conductor are recorded in \cite{EK_K3SAEFHMS}. We record a (non-exhaustive) list of such curves with $\sqrt{N_J} \leq \condbound$ in \cite{ME_ELECTRONIC_here}. Note that when $D > 17$ the moduli space of curves with RM by $\mathcal{O}_D$ is not rational, so examples are sparser than when $D = 8$.

We run the algorithm in \Cref{sec:algo-twist-X7} for each curve recorded in \cite{ME_ELECTRONIC_here}.

\begin{remark}
  It would be interesting to compute the twists $X_{\mathcal{J}[\ffp]}^{\pm}(7)$ for the Jacobian $\mathcal{J}$ of the generic curves $\mathcal{C}/\bbQ(a,b,c)$ with RM by $\mathcal{O}_D$ given in \cite{B_COG2W2M_thesis,B_COG2W2M} and \cite{CFM_FMFG2CWRM} for each $D = 8$, $29$, $37$, $44$, and $53$ (i.e., those $D$ where $7$ splits in $\mathcal{O}_D$ and for which \cite{CFM_FMFG2CWRM} gives a generic model for a curve by $\mathcal{O}_D$). Unfortunately the algorithm we describe is ill suited to this task. One might hope to interpolate over twists computed for a large number of specialisations. However these twists are only defined up to the action of $\Aut_{\bbQ}(\bbP^2) \cong \PGL_3(\bbQ)$ and our algorithm for generating matrices which satisfy Hilbert's Theorem~90 does not do so in a compatible way (it requires a choice of $\ffp | 7$, a choice of basis for $J[\ffp]$, and a ``randomly'' generated matrix).
\end{remark}

\section{Proving twists of \texorpdfstring{$X(7)$}{X(7)} are isomorphic to \texorpdfstring{$X_M^\pm(7)$}{X\_M(7)}}
\label{sec:proving-twists}
Let $M/\bbQ$ be an irreducible $G_{\bbQ}$-module and let $\mathcal{X}/\bbQ$ be a plane quartic curve (in our case we will take $M = J[\ffp]$ and $\mathcal{X}$ to be an output of the algorithm in \Cref{sec:algo-twist-X7}). We now outline an approach for proving that $\mathcal{X}$ is isomorphic to a twist $X_M^\pm(7)$ of the Klein quartic (for some choice of sign). We assume that $\mathcal{X}$ is a twist of $X(7)$ (note that this is simple to check by computing Dixmier--Ohno invariants~\cite{D_OTPIOQPC,O_TGROIOTQI,E_ECOIOPQC} in \texttt{Magma}).

For the purpose of proving \Cref{coro:lots-of-sha} it suffices to consider only the case when $\mathcal{X}$ has a rational point (i.e., it suffices to recall \cite[Lemma~6.2]{F_VEOO7ITTSGOAEC}, see \Cref{lemma:qt-isom} below). In \Cref{sec:when-X-has-no-pts} we note how one may prove that a twist is isomorphic to $X_M^\pm(7)$ more generally.

\subsection{When $\mathcal{X}$ has a rational point}
\label{sec:when-X-has-pt}
Let $K$ be a field of characteristic zero. Suppose that $\mathcal{X}$ has a $K$-rational point which corresponds (through the moduli interpretation in \Cref{lemma:calX-moduli}) to an elliptic curve $E/K$ (defined up to quadratic twist) with $j$-invariant $j(E) \neq 0, 1728, \infty$. In this case, the following lemma reduces the problem of showing that $\mathcal{X}$ is isomorphic to $X_M^{\pm}(7)$ to the problem of showing that $X_E(7)$ is isomorphic to $X_M^{\pm}(7)$. 

\begin{lemma}
  Let $\mathcal{X}/K$ be a twist of $X(7)$ and suppose that there exists a point $P \in \mathcal{X}(K)$ with $j(P) \neq 0,1728,\infty$. If $E/K$ is an elliptic curve with $j(E) = j(P)$ then $\mathcal{X}$ is isomorphic to $X_E(7)$ over $K$.
\end{lemma}

\begin{proof}
  Let $\varphi \colon \mathcal{X} \cong X_E(7)$ be a $\Kbar$-isomorphism. By composing with a $\Kbar$-automorphism of $X_E(7)$ we may assume that $\varphi(P)$ is equal to the tautological point $Q = (E, \mathrm{id}) \in X_E(7)(K)$. Since $P$ and $Q$ are $K$-rational, for each $\sigma \in G_K$ we have $\sigma \varphi \sigma^{-1}(P) = Q$, so that $\varphi^{-1} \sigma \varphi \sigma^{-1} (P) = P$ for each $\sigma \in G_K$. Since $j(P) \neq 0, 1728, \infty$ the only $\Kbar$-automorphism of $\mathcal{X}$ which fixes $P$ is the identity. In particular, $\varphi = \sigma \varphi \sigma^{-1}$ for all $\sigma \in G_K$ and therefore $\varphi$ is defined over $K$.
\end{proof}

To show that $X_E(7)$ is isomorphic to $X_M^\pm(7)$ for some choice of sign, it suffices to show that $E[7]$ is isomorphic to $M$ as a $G_\bbQ$-module, up to quadratic twist. We recall the following lemma of Fisher (based on an argument of Serre using Goursat's lemma~\cite[Lemme~8]{S_PGDPDFDOCE}) which allows us to prove such congruences, up to quadratic twists. 

\begin{lemma}
  \label{lemma:qt-isom}
  Let $K$ be a number field and let $M$ be a $G_{K}$-module which is isomorphic as an abstract group to $(\bbZ/p\bbZ)^2$ for some $p \geq 5$. Suppose that $M$ comes equipped with a ($G_K$-equivariant) alternating pairing $M \times M \to \mu_p$. Let $E/K$ be an elliptic curve with surjective mod $p$ Galois representation, let $x_M \colon M \to M/\{\pm 1\}$, and let $x_E \colon E \to \bbP^1$ be the quotient by $\{\pm 1\}$. If there exist non-identity elements $P \in M$ and $Q \in E[p]$ such that $K(x_M(P)) = K(x_E(Q))$ then there exists a quadratic twist $E^d$ of $E$ such that $M \cong E^d[p]$.  
\end{lemma}

\begin{proof}
  This follows immediately from \cite[Lemma~6.2]{F_VEOO7ITTSGOAEC} (cf. \cite[Proposition~6.1]{F_VEOO7ITTSGOAEC}). Note that the hypothesis that $K = \bbQ$ in \cite[Lemma~6.2]{F_VEOO7ITTSGOAEC} is not used.
\end{proof}

\subsection{When $\mathcal{X}$ has no rational points}
\label{sec:when-X-has-no-pts}
We rely on the approach in \Cref{sec:when-X-has-pt} together with the following criterion. In practice when $X$ and $Y$ are twists of $X(7)$ defined over $\bbQ$ it is simple to find number fields for which the statement holds. In this case there exist infinitely many points on $X$ and $Y$ defined over quartic fields. One expects that if $X$ and $Y$ are isomorphic (and have no non-trivial automorphisms defined over $\bbQ$), then for a generic such field the conditions of the lemma hold.

\begin{lemma}
  \label{lemma:X-twist-no-pts}
  Let $X/K$ and $Y/K$ be (geometrically integral) curves defined over a number field $K$. Suppose that there exist extensions $L_1,L_2/K$ for which $L_1 \cap L_2 = K$ and such that we have isomorphisms $\varphi_i \colon X_{L_i} \cong Y_{L_i}$ for each $i = 1, 2$. If $X$ (or $Y$) does not admit a non-trivial automorphism over the compositum $L_1 L_2$, then $X$ and $Y$ are isomorphic over $K$.
\end{lemma}

\begin{proof}
  The assumption on the automorphism group of $X$ over $L_1 L_2$ implies that the composition $\varphi_1^{-1} \varphi_2$ is the identity, and therefore over $L_1L_2$ we have an equality $\varphi_1 = \varphi_2$. But then $\varphi_1$ is defined over $L_1 \cap L_2 = K$ and the claim follows.
\end{proof}

\begin{prop}
  \label{prop:XM-proof}
  Consider any of the data in \cite[\texttt{data/twists.m}]{ME_ELECTRONIC_here} which consists of
  \begin{enumerate}[label=\eniii]
  \item
    a genus $2$ curve $C/\bbQ$, 
  \item
    a fundamental discriminant $D > 0$ such that the Jacobian of $C$ has RM by $\mathcal{O}_D$, and
  \item
    a twist $\mathcal{X}/\bbQ$ of the Klein quartic.
  \end{enumerate}
  Then $\mathcal{X}$ is isomorphic over $\bbQ$ to $X_{J[\ffp]}^\pm(7)$ for some choice of sign and choice of prime $\ffp \subset \mathcal{O}_D$ above $7$.
\end{prop}

\begin{proof}
  If $\mathcal{X}$ has a $\bbQ$-rational point of small height corresponding to an elliptic curve $E/\bbQ$, we apply \Cref{lemma:qt-isom}. In each case, applying \cite[Proposition~19]{S_PGDPDFDOCE} at several good primes (or using Zywina's algorithm~\cite{Z_EOIFECOQ}) suffices to show that the mod $7$ Galois representation attached to $E/\bbQ$ is surjective. In the electronic data we exhibit an explicit isomorphism between the fields $x_J(P)$ and $x_E(Q)$ for some $P \in J[\ffp]$ and $Q \in E[7]$ (note that a minimal polynomial for the extension $\bbQ(x_J(P))/\bbQ$ was computed in the course of the algorithm in \Cref{sec:an-expl-descr}).

  The general case proceeds similarly. Taking hyperplane sections of $\mathcal{X}$ we construct non-isomorphic quartic fields $L_1,L_2/\bbQ$ over which $\mathcal{X}$ obtains a point and such that $L_i$ contains no non-trivial subfield for each $i = 1,2$ (in particular $L_1 \cap L_2 = \bbQ$ and $L_1 \cap \bbQ(\zeta_7) = L_2 \cap \bbQ(\zeta_7) = \bbQ$). These points correspond to elliptic curves $E_1/L_1$ and $E_2/L_2$ whose mod $7$ Galois representations may be seen to be surjective by applying \cite[Proposition~19]{S_PGDPDFDOCE} at several places of good reduction. Applying \Cref{lemma:qt-isom} as above shows that $E_1$ and $E_2$ are $(7, \ffp)$-congruent to $J$, up to a quadratic twist. Since the mod $7$ Galois representations of $E_1/L_1$ and $E_2/L_2$ are surjective we have $\bbQ(\mathcal{K}[\ffp]) \cap L_1 L_2 = \bbQ$, where $\mathcal{K} = J/\{\pm 1\}$ is the Kummer surface of $J$. 

  It follows from the construction that $\Aut(X_{J[\ffp]}^{r}(7))$ is isomorphic (as a $G_\bbQ$-module) to the group $\Aut_r(J[\ffp])/\{\pm 1\}$ consisting of automorphisms of $J[\ffp]$ which are symplectic with respect to $(e_{J,7})^{r}$. Therefore, the field of definition of the automorphisms of $X_{J[\ffp]}^\pm(7)$ is equal to $\bbQ(\mathcal{K}[\ffp])$ and $X_{J[\ffp]}^\pm(7)$ admits no non-trivial automorphisms over $L_1L_2$ that are not defined over $\bbQ$. Suppose there is such an automorphism $\tau$ defined over $\bbQ$. Since the mod $7$ Galois representation attached to $J[\ffp]$ is surjective ($J[\ffp]$ is isomorphic over $L_i$ to a quadratic twist of $E_i[7]$ for each $i = 1,2$) the element $\tau$ is contained in the centre of $\Aut_r(J[\ffp])/\{ \pm 1\}$ which is isomorphic to $\PSL_2(\bbZ/7\bbZ)$ as an abstract group. Therefore $\tau$ is the identity and the claim follows from \Cref{lemma:X-twist-no-pts}.
\end{proof}

\section{Proving \texorpdfstring{$(7, \ffp)$}{(7, p)}-congruences and \texorpdfstring{\Cref{coro:lots-of-sha}}{Theorem 1}}
\label{sec:proving-congs}
We now prove \Cref{coro:lots-of-sha}. In order to apply visibility we must first show that the pairs $(C,E)$ in \Cref{tab:sha-exs} are in fact $(7, \ffp)$-congruent (not simply up to quadratic twist, as we proved in \Cref{sec:proving-twists}).

\begin{lemma}
  \label{lemma:where-qt}
  Let $E/\bbQ$ be an elliptic curve and let $J/\bbQ$ be a genus $2$ Jacobian with RM by $\mathcal{O}_D$. Suppose that $(p) = \ffp\bar{\ffp}$ in $\mathcal{O}_D$ and that there exists a squarefree integer $d \in \bbZ$ such that $E^d$ and $J$ are $(p, \ffp)$-congruent. Then $d$ is supported on the set of primes consisting of $p$, the bad primes of $E$, and the bad primes of $J$.
\end{lemma}

\begin{proof}
  This is similar to \cite[Proposition~4.18]{FK_OTSTOIOTPTOEC} and \cite[Lemma~3.6]{F_O12COEC} (see also \cite[Lemme~8]{S_PGDPDFDOCE}). Let $\ell \neq p$ be a prime at which $J$ has good reduction and at which $E$ has potentially good reduction. Let $\bbQ_\ell^{\textrm{ur}}$ be the maximal unramified extension of $\bbQ_\ell$ and let $K = \bbQ_{\ell}^{\mathrm{ur}}(J[p])$. By \cite[\S2~Corollary~3]{ST_GROAV} if $A/\bbQ_{\ell}$ is an abelian variety with potential good reduction at $\ell$, then for each $p \neq \ell$ the field $\bbQ_\ell^{\mathrm{ur}}(A[p])$ is the smallest extension of $\bbQ_\ell^{\mathrm{ur}}$ over which $A$ attains good reduction. But then we have $\bbQ_{\ell}^{\mathrm{ur}}(E[p]) = \bbQ_{\ell}^{\mathrm{ur}}(J[\ffp]) \subset K = \bbQ_{\ell}^{\mathrm{ur}}$, as required.
\end{proof}

\begin{prop}
  \label{thm:sqrt2-congs}
  For each pair $(E, C)$ of elliptic curve $E/\bbQ$ and genus 2 curve $C/\bbQ$ in \Cref{tab:sha-exs} we have a $(7, \ffp)$-congruence between $E$ and $J = \Jac(C)$ for some choice of $\ffp \vert 7$ in $\mathcal{O}_D$.
\end{prop}

\begin{proof}
  Let $\ell \neq 7$ be a good prime for $C$ and $E$. By \cite[(5.2)]{F_VEOO7ITTSGOAEC} (which follows from \cite[Section~2.1]{FLSSSW_EEFTBSDCFMJOG2C} or \cite[Lemma~3]{MS_COG2WGRAF2WARWP}) a $(7,\ffp)$-congruence between $E^d/\bbQ$ and $J/\bbQ$ gives a congruence modulo $7$
  \begin{equation}
    \label{eq:trace-cond}
    a_{\ell}(E^d)^2 - t_\ell a_\ell(E^d) + n_\ell \equiv 0 \pmod{7}
  \end{equation}
  where $t_\ell = \ell + 1 - N_1 $ and $n_\ell = (N_1^2 + N_2 )/2 - (\ell + 1) N_1 - \ell$ where $N_1 = \# C(\mathbb{F}_\ell)$ and $N_2 = \# C(\mathbb{F}_{\ell^2})$.

  Testing \eqref{eq:trace-cond} on the divisors $d$ of the product of $7$ and the bad primes of $E$ and $C$ shows that $E^d$ and $J$ are not $(7,\ffp)$-congruent for any $d \neq 1$ (by \Cref{lemma:where-qt}). By \Cref{prop:XM-proof} $E$ and $J$ are $(7,\ffp)$-congruent up to quadratic twist since $E$ corresponds to a point on one of the twists $X_{J[\ffp]}^\pm(7)$. It therefore follows that $E$ and $J$ are $(7, \ffp)$-congruent for some choice of $\ffp$ dividing $7$.
\end{proof}

Using the congruences supplied by \Cref{thm:sqrt2-congs} we now prove \Cref{coro:lots-of-sha,thm:sha7-discs} by applying \cite[Theorem~2.2]{F_VEOO7ITTSGOAEC}.

\begin{proof}[Proof of \Cref{coro:lots-of-sha,thm:sha7-discs}]
  This follows from \cite[Theorem~2.2]{F_VEOO7ITTSGOAEC}, as we detail below.
  
  Let $E/\bbQ$ and $J/\bbQ$ be one of the pairs of elliptic curve and genus $2$ Jacobian from \Cref{coro:lots-of-sha} or \ref{thm:sha7-discs}. We check that in each case $J$ is geometrically simple by applying the condition in \cite[Section~14.4]{CF_PTAMAOCOG2} and \cite{S_TS2DAVDOQWMWGORAL19}. The $7$-torsion subgroups of $E(\bbQ)$ and $J(\bbQ)$ are trivial. The rank of $E/\bbQ$ is $2$ and the rank of $J/\bbQ$ is $0$ (the rank of $J/\bbQ$ is bounded using $2$-descent, which is implemented as \texttt{RankBounds} in \texttt{Magma}). For each discriminant $D$ appearing in \Cref{coro:lots-of-sha,thm:sha7-discs} the prime $7$ not only splits in $\mathcal{O}_D$, but $7 = \Nm \eta$ for some $\eta \in \mathcal{O}_D$. In particular the isogeny $\ffp$ is equal to the multiplication-by-$\eta$-map on $J$, and $J(\bbQ)/\ffp J(\bbQ) = 0$.

  By \Cref{thm:sqrt2-congs} the elliptic curve $E$ is $(7,\ffp)$-congruent to $J$. The abelian varieties $E$ and $J$ have good reduction at $7$, so by \cite[Theorem~2.2]{F_VEOO7ITTSGOAEC} it suffices to show that the Tamagawa numbers of $E/\bbQ$ and $J/\bbQ$ are coprime to $7$.
  
  We compute the Tamagawa numbers of $E/\bbQ$ using \texttt{Magma}. Except for the Jacobian of conductor $3200^2$ in \Cref{tab:sha-exs}, for each bad prime $p$ of $J$ one may check that the order of the geometric component group of $J$ at $p$ is coprime to $7$ using Liu's \texttt{genus2reduction} in \texttt{SageMath} and Donnelly's \texttt{Magma} functions \texttt{RegularModel} and \texttt{ComponentGroup}.

  For the Jacobian of conductor $3200^2$ in \Cref{tab:sha-exs} the computation of the Tamagawa number of $J/\bbQ$ at $2$ was carried out in the appendix to \cite{KS_CVOSBSDFMMASOQ} (where it is shown that the Tamagawa number is $1$). 
\end{proof}

\section{Evidence towards \texorpdfstring{\Cref{conj:sha-invisible}}{Conjecture 6}}
\label{sec:worthy-example}
Consider the genus $2$ curve $C/\bbQ$ with LMFDB label \LMFDBLabelGenusTwo{385641.a.385641.1} and Weierstrass equation
\begin{equation*}
  C : y^2 + (x^3 + 1)y = -6x^4 + 6x^3 + 27x^2 - 30x - 22.
\end{equation*}
The Jacobian $J/\bbQ$ of $C$ has RM by $\bbZ[\sqrt{2}]$. In \cite[A.3]{KS_CVOSBSDFMMASOQ} it is noted that the Birch and Swinnerton-Dyer conjecture predicts $| \Sha(J^{-11}/\bbQ) | = 7^2$. By \Cref{prop:XM-proof}, for some choice of factorisation $(7) = \ffp \bar{\ffp}$ in $\bbZ[\sqrt{2}]$ we have models
\begingroup
\allowdisplaybreaks
\begin{align*}
  X_{J[\ffp]}^{\pm}(7) :&\; -2 x_0^4 + 39 x_0^3 x_1 + 11 x_0^3 x_2 - 42 x_0^2 x_1^2 - 18 x_0^2 x_1 x_2 + 20 x_0 x_1^3 - 6 x_0 x_1^2 x_2 \\
              &+ 12 x_0 x_1 x_2^2 - 7 x_0 x_2^3 - 24 x_1^4 + 13 x_1^3 x_2 + 15 x_1^2 x_2^2 + 9 x_1 x_2^3 + x_2^4  = 0, \\
  X_{J[\ffp]}^{\mp}(7) :&\; 2 x_0^4 + 5 x_0^3 x_1 + 9 x_0^3 x_2 + 6 x_0^2 x_2^2 - x_0 x_1^3 - 6 x_0 x_1^2 x_2 + 12 x_0 x_1 x_2^2 + 2 x_0 x_2^3 \\
                              &- x_1^4 - 3 x_1^3 x_2 + 3 x_1^2 x_2^2 + 17 x_1 x_2^3 + 12 x_2^4 = 0, \\
  X_{J[\bar{\ffp}]}^{\pm}(7) :&\; x_0^4 - 3 x_0^3 x_1 - 28 x_0^3 x_2 - 15 x_0^2 x_1^2 - 3 x_0^2 x_1 x_2 + 39 x_0^2 x_2^2 - 6 x_0 x_1^3 - 12 x_0 x_1^2 x_2 \\
              & - 6 x_0 x_1 x_2^2 - 29 x_0 x_2^3 + 3 x_1^4 + 9 x_1^3 x_2 + 30 x_1^2 x_2^2 - 3 x_1 x_2^3 - 10 x_2^4 = 0, \\
  X_{J[\bar{\ffp}]}^{\mp}(7) :&\; -4 x_0^4 + 6 x_0^3 x_1 + 7 x_0^3 x_2 + 3 x_0^2 x_1^2 + 12 x_0 x_1^3 + 6 x_0 x_1^2 x_2 - 9 x_0 x_1 x_2^2 -  x_0 x_2^3  \\
              & - 6 x_1^4 - 3 x_1^3 x_2 + 3 x_1^2 x_2^2 + 6 x_1 x_2^3 + x_2^4 = 0.
\end{align*}
\endgroup
We were unable to find rational points on any of these curves, except on $X_{J[\ffp]}^{\mp}(7)$ where we find exactly one point which corresponds to the elliptic curve $E/\bbQ$ with LMFDB label \LMFDBLabel{1242.m1} and Weierstrass equation $y^2 + xy + y = x^3 - x^2 - 1666739x - 2448131309$. Using the argument in \Cref{thm:sqrt2-congs} it can be shown that $E$ and $J$ are $(7,\ffp)$-congruent. However, the quadratic twist of $E$ by $-11$ has trivial Mordell-Weil group, so cannot be used to visualise the (conjectural) non-trivial elements of $\Sha(J^{-11}/\bbQ)[7]$.

\newcommand{\etalchar}[1]{$^{#1}$}
\providecommand{\bysame}{\leavevmode\hbox to3em{\hrulefill}\thinspace}
\providecommand{\MR}{\relax\ifhmode\unskip\space\fi MR }
\providecommand{\MRhref}[2]{%
  \href{http://www.ams.org/mathscinet-getitem?mr=#1}{#2}
}
\providecommand{\bibtitleref}[2]{%
  \hypersetup{urlbordercolor=0.8 1 1}%
  \href{#1}{#2}%
  \hypersetup{urlbordercolor=cyan}%
}
\providecommand{\href}[2]{#2}


\begin{thebibliography}{{Zyw}22}

\bibitem[AS02]{AS_VOSTGOAV}
A.~{Agashe} and W.~{Stein},
  \bibtitleref{https://doi.org/10.1006/jnth.2002.2810}{\emph{Visibility of
  {S}hafarevich-{T}ate groups of abelian varieties}}, J. Number Theory
  \textbf{97} (2002), no.~1, 171--185. \MR{1939144}.

\bibitem[AS05]{AS_VEFTBSDCFMAVOAR0}
\bysame,
  \bibtitleref{https://doi.org/10.1090/S0025-5718-04-01644-8}{\emph{Visible
  evidence for the {B}irch and {S}winnerton-{D}yer conjecture for modular
  abelian varieties of analytic rank zero}}, Math. Comp. \textbf{74} (2005),
  no.~249, 455--484, With an appendix by J. Cremona and B. Mazur. \MR{2085902}.

\bibitem[BCP97]{MAGMA}
W.~{Bosma}, J.~{Cannon}, and C.~{Playoust},
  \bibtitleref{http://dx.doi.org/10.1006/jsco.1996.0125}{\emph{The {M}agma
  algebra system. {I}. {T}he user language}}, J. Symbolic Comput. \textbf{24}
  (1997), no.~3-4, 235--265, Computational algebra and number theory (London,
  1993). \MR{1484478}.

\bibitem[{Ben}98]{B_COG2W2M_thesis}
P.~R. {Bending}, \emph{Curves of genus $2$ with $\sqrt{2}$-multiplication},
  Ph.D. thesis, University of Oxford, 1998.

\bibitem[{Ben}99]{B_COG2W2M}
\bysame,
  \bibtitleref{https://doi.org/10.48550/arXiv.math/9911273}{\emph{{Curves of
  genus 2 with sqrt2 multiplication}}}, arXiv e-prints (1999), math/9911273.

\bibitem[BFT14]{BFT_DV33IOJOG2C}
N.~{Bruin}, E.~V. {Flynn}, and D.~{Testa},
  \bibtitleref{https://doi.org/10.4064/aa165-3-1}{\emph{Descent via
  {$(3,3)$}-isogeny on {J}acobians of genus 2 curves}}, Acta Arith.
  \textbf{165} (2014), no.~3, 201--223. \MR{3263947}.

\bibitem[BKLS21]{BKOS_EOGOITSGOAVIQTF}
M.~{Bhargava}, Z.~{Klagsbrun}, R.~J. {Lemke Oliver}, and A.~{Shnidman},
  \bibtitleref{https://doi.org/10.2140/ant.2021.15.627}{\emph{Elements of given
  order in {T}ate-{S}hafarevich groups of abelian varieties in quadratic twist
  families}}, Algebra Number Theory \textbf{15} (2021), no.~3, 627--655.
  \MR{4261095}.

\bibitem[CF96]{CF_PTAMAOCOG2}
J.~W.~S. {Cassels} and E.~V. {Flynn},
  \bibtitleref{https://doi.org/10.1017/CBO9780511526084}{\emph{Prolegomena to a
  middlebrow arithmetic of curves of genus {$2$}}}, London Mathematical Society
  Lecture Note Series, vol. 230, Cambridge University Press, Cambridge, 1996.
  \MR{1406090}.

\bibitem[CFM24]{CFM_FMFG2CWRM}
A.~{Cowan}, S.~{Frengley}, and K.~{Martin},
  \bibtitleref{https://doi.org/10.48550/arXiv.2403.03191}{\emph{Generic models
  for genus 2 curves with real multiplication}}, arXiv e-prints (2024),
  arXiv:2403.03191.

\bibitem[CM00]{CM_VEITSTG}
J.~E. {Cremona} and B.~{Mazur},
  \bibtitleref{http://projecteuclid.org/euclid.em/1046889588}{\emph{Visualizing
  elements in the {S}hafarevich-{T}ate group}}, Experiment. Math. \textbf{9}
  (2000), no.~1, 13--28. \MR{1758797}.

\bibitem[DD19]{DD_3TACOG2C}
T.~{Dokchitser} and C.~{Doris},
  \bibtitleref{https://doi.org/10.1090/mcom/3387}{\emph{3-torsion and conductor
  of genus 2 curves}}, Math. Comp. \textbf{88} (2019), no.~318, 1913--1927.
  \MR{3925491}.

\bibitem[{Dix}87]{D_OTPIOQPC}
J.~{Dixmier},
  \bibtitleref{https://doi.org/10.1016/0001-8708(87)90010-7}{\emph{On the
  projective invariants of quartic plane curves}}, Adv. in Math. \textbf{64}
  (1987), no.~3, 279--304. \MR{888630}.

\bibitem[EK14]{EK_K3SAEFHMS}
N.~D. {Elkies} and A.~{Kumar},
  \bibtitleref{https://doi-org/10.2140/ant.2014.8.2297}{\emph{K3 surfaces and
  equations for {H}ilbert modular surfaces}}, Algebra Number Theory \textbf{8}
  (2014), no.~10, 2297--2411. \MR{3298543}.

\bibitem[{Elk}99]{E_TKQINT}
N.~D. {Elkies}, \emph{The {K}lein quartic in number theory}, The eightfold way,
  Math. Sci. Res. Inst. Publ., vol.~35, Cambridge Univ. Press, Cambridge, 1999,
  pp.~51--101. \MR{1722413}.

\bibitem[{Els}15]{E_ECOIOPQC}
A.-S. {Elsenhans},
  \bibtitleref{https://doi.org/10.1016/j.jsc.2014.09.006}{\emph{Explicit
  computations of invariants of plane quartic curves}}, J. Symbolic Comput.
  \textbf{68} (2015), 109--115. \MR{3283857}.

\bibitem[ES24]{ES_MOH}
{A.-S.} Elsenhans and {M.} Stoll,
  \bibtitleref{https://doi.org/10.1090/mcom/3924}{\emph{Minimization of
  hypersurfaces}}, Math. Comp. \textbf{93} (2024), no.~349, 2513--2555.
  \MR{4759383}.

\bibitem[{Fis}14]{F_OFO7A11CEC}
T.~A. {Fisher},
  \bibtitleref{https://doi.org/10.1112/S1461157014000059}{\emph{On families of
  7- and 11-congruent elliptic curves}}, LMS J. Comput. Math. \textbf{17}
  (2014), no.~1, 536--564. \MR{3356045}.

\bibitem[{Fis}16]{F_VEOO7ITTSGOAEC}
\bysame,
  \bibtitleref{https://doi.org/10.1112/S1461157016000243}{\emph{Visualizing
  elements of order 7 in the {T}ate-{S}hafarevich group of an elliptic curve}},
  LMS J. Comput. Math. \textbf{19} (2016), no.~suppl. A, 100--114.
  \MR{3540949}.

\bibitem[FK22]{FK_OTSTOIOTPTOEC}
N.~{Freitas} and A.~{Kraus},
  \bibtitleref{https://doi-org/10.1090/memo/1361}{\emph{On the symplectic type
  of isomorphisms of the {$p$}-torsion of elliptic curves}}, Mem. Amer. Math.
  Soc. \textbf{277} (2022), no.~1361, v+105. \MR{4403927}.

\bibitem[FLS{\etalchar{+}}01]{FLSSSW_EEFTBSDCFMJOG2C}
E.~V. {Flynn}, F.~{Lepr\'{e}vost}, E.~F. {Schaefer}, W.~A. {Stein}, M.~{Stoll},
  and J.~L. {Wetherell},
  \bibtitleref{https://doi.org/10.1090/S0025-5718-01-01320-5}{\emph{Empirical
  evidence for the {B}irch and {S}winnerton-{D}yer conjectures for modular
  {J}acobians of genus 2 curves}}, Math. Comp. \textbf{70} (2001), no.~236,
  1675--1697. \MR{1836926}.

\bibitem[{Fly}15]{F_DV55IOJOG2C}
E.~V. {Flynn},
  \bibtitleref{https://doi.org/10.1016/j.jnt.2015.01.018}{\emph{Descent via
  {$(5,5)$}-isogeny on {J}acobians of genus 2 curves}}, J. Number Theory
  \textbf{153} (2015), 270--282. \MR{3327574}.

\bibitem[{Fre}]{ME_ELECTRONIC_here}
S.~{Frengley}, \emph{Github repository},
  \url{https://github.com/SamFrengley/sha-7-examples.git}.

\bibitem[{Fre}23]{F_thesis}
\bysame,
  \bibtitleref{https://www.repository.cam.ac.uk/handle/1810/369243}{\emph{Explicit
  moduli spaces for curves of genus 1 and 2}}, Ph.D. thesis, University of
  Cambridge, 2023, \url{https://www.repository.cam.ac.uk/handle/1810/369243}.

\bibitem[{Fre}24]{F_O12COEC}
\bysame, \bibtitleref{https://doi.org/10.1142/S1793042124500301}{\emph{On
  $12$-congruences of elliptic curves}}, Int. J. Number Theory \textbf{20}
  (2024), no.~2, 565--601. \MR{4709643}.

\bibitem[FS25]{FS_ALpTITSG}
E.~V. Flynn and A.~Shnidman,
  \bibtitleref{https://doi.org/10.1017/S1474748024000392}{\emph{Arbitrarily
  large {{\(p\)}}-torsion in {Tate}-{Shafarevich} groups}}, J. Inst. Math.
  Jussieu \textbf{24} (2025), no.~2, 481--502.

\bibitem[{Kle}78]{K_UDTSODEF}
F.~{Klein}, \bibtitleref{https://doi.org/10.1007/BF01677143}{\emph{Ueber die
  {T}ransformation siebenter {O}rdnung der elliptischen {F}unctionen}}, Math.
  Ann. \textbf{14} (1878), no.~3, 428--471. \MR{1509988}.

\bibitem[KS25]{KS_CVOSBSDFMMASOQ}
T.~Keller and M.~Stoll,
  \bibtitleref{https://doi.org/10.1017/fms.2024.133}{\emph{Complete
  verification of strong {BSD} for many modular abelian surfaces over
  {{\(\mathbf{Q}\)}}}}, Forum Math. Sigma \textbf{13} (2025), 82,
  Id/No e20.

\bibitem[KW09a]{KW_SMCI}
C.~{Khare} and J.-P. {Wintenberger},
  \bibtitleref{https://doi.org/10.1007/s00222-009-0205-7}{\emph{Serre's
  modularity conjecture. {I}}}, Invent. Math. \textbf{178} (2009), no.~3,
  485--504. \MR{2551763}.

\bibitem[KW09b]{KW_SMCII}
\bysame, \bibtitleref{https://doi.org/10.1007/s00222-009-0206-6}{\emph{Serre's
  modularity conjecture. {II}}}, Invent. Math. \textbf{178} (2009), no.~3,
  505--586. \MR{2551764}.

\bibitem[LMFDB]{lmfdb}
The {LMFDB Collaboration}, \emph{The {L}-functions and modular forms database},
  \url{http://www.lmfdb.org}, 2025, [Online; accessed 20 January 2025].

\bibitem[MS93]{MS_COG2WGRAF2WARWP}
J.~R. {Merriman} and N.~P. {Smart},
  \bibtitleref{https://doi.org/10.1017/S030500410007153X}{\emph{Curves of genus
  {$2$} with good reduction away from {$2$} with a rational {W}eierstrass
  point}}, Math. Proc. Cambridge Philos. Soc. \textbf{114} (1993), no.~2,
  203--214. \MR{1230127}.

\bibitem[{Nic}]{N_Electronic}
C.~{Nicholls}, \emph{Github repository},
  \url{https://github.com/cgnicholls/phd-code.git}, Accessed November 2023.

\bibitem[{Nic}18]{N_DMATOJOHGC}
\bysame,
  \bibtitleref{https://ora.ox.ac.uk/objects/uuid:04cef70a-2ab9-44c2-8bbe-ca2ac33bfe41}{\emph{Descent
  methods and torsion on {J}acobians of higher genus curves}}, Ph.D. thesis,
  University of Oxford, 2018.

\bibitem[{Ohn}]{O_TGROIOTQI}
T.~{Ohno},
  \bibtitleref{https://aeb.win.tue.nl/math/ohno-preprint.2007.05.15.pdf}{\emph{The
  graded ring of invariants of ternary quartics {I}}}, Preprint,
  \url{https://aeb.win.tue.nl/math/ohno-preprint.2007.05.15.pdf}.

\bibitem[PSS07]{PSS_TOX7APSTX2Y3Z7}
B.~{Poonen}, E.~F. {Schaefer}, and M.~{Stoll},
  \bibtitleref{https://doi.org/10.1215/S0012-7094-07-13714-1}{\emph{Twists of
  {$X(7)$} and primitive solutions to {$x^2+y^3=z^7$}}}, Duke Math. J.
  \textbf{137} (2007), no.~1, 103--158. \MR{2309145}.

\bibitem[Sage]{sagemath}
{The Sage Developers}, \emph{{S}agemath, the {S}age {M}athematics {S}oftware
  {S}ystem ({V}ersion 9.5)}, 2022, \url{https://www.sagemath.org}.

\bibitem[{Ser}72]{S_PGDPDFDOCE}
J.-P. {Serre},
  \bibtitleref{https://doi.org/10.1007/BF01405086}{\emph{Propri\'{e}t\'{e}s
  galoisiennes des points d'ordre fini des courbes elliptiques}}, Invent. Math.
  \textbf{15} (1972), no.~4, 259--331. \MR{387283}.

\bibitem[{Ser}79]{S_LF}
\bysame, \bibtitleref{https://doi.org/10.1007/978-1-4757-5673-9}{\emph{Local
  fields}}, Springer-Verlag, New York-Berlin, 1979, Translated from the French
  by Marvin Jay Greenberg. \MR{554237}.

\bibitem[ST68]{ST_GROAV}
J.-P. {Serre} and J.~{Tate},
  \bibtitleref{https://doi.org/10.2307/1970722}{\emph{Good reduction of abelian
  varieties}}, Ann. of Math. (2) \textbf{88} (1968), 492--517. \MR{236190}.

\bibitem[{Sto}95]{S_TS2DAVDOQWMWGORAL19}
M.~{Stoll}, \emph{Two simple {$2$}-dimensional abelian varieties defined over
  {${\bf Q}$} with {M}ordell-{W}eil group of rank at least {$19$}}, C. R. Acad.
  Sci. Paris S\'{e}r. I Math. \textbf{321} (1995), no.~10, 1341--1345.
  \MR{1363577}.

\bibitem[SW22]{SW_EOPOITSGOAVOQ}
A.~{Shnidman} and A.~{Weiss},
  \bibtitleref{https://doi.org/10.1017/fms.2022.80}{\emph{Elements of prime
  order in {T}ate-{S}hafarevich groups of abelian varieties over {$\Bbb{Q}$}}},
  Forum Math. Sigma \textbf{10} (2022), Paper No. e98, 10. \MR{4504870}.

\bibitem[{Zyw}22]{Z_EOIFECOQ}
D.~{Zywina},
  \bibtitleref{https://doi.org/10.48550/arXiv.2206.14959}{\emph{Explicit open
  images for elliptic curves over $\mathbb{Q}$}}, arXiv e-prints (2022),
  arXiv:2206.14959.

\end{thebibliography}
\end{document}